\documentclass[11pt]{amsart}
\pdfoutput=1

\usepackage[usenames,dvipsnames]{xcolor}   %

\usepackage[utf8]{inputenc}
\usepackage[T1]{fontenc}

\usepackage[letterpaper,margin=1in]{geometry}

\usepackage{lmodern}
\usepackage{microtype}
\usepackage{etoolbox}                 %
\usepackage{lettrine}

\usepackage{mathtools}
\usepackage{amsmath,amssymb,amsfonts,amsxtra,bm}
\usepackage{amsthm}
\usepackage{mathrsfs}
\usepackage{latexsym}
\usepackage{nicefrac}       %

\usepackage{booktabs}
\usepackage{makecell,multirow}       %
\usepackage{colortbl}
\usepackage{graphicx}
\usepackage{epstopdf}
\usepackage{float}
\floatstyle{plaintop}
\restylefloat{table}
\usepackage{footnote, tablefootnote}

\usepackage{enumitem}
\usepackage{algorithm}
\usepackage{algorithmic}
\usepackage{comment}
\usepackage{xspace}
\usepackage{url}
\usepackage[most]{tcolorbox}

\usepackage[round]{natbib}

\usepackage{hyperref}

\numberwithin{equation}{section}

\makeatletter

\definecolor{darkblue}{rgb}{0.0,0.0,0.65}
\definecolor{darkred}{rgb}{0.68,0.05,0.0}
\definecolor{darkgreen}{rgb}{0.0,0.29,0.29}
\definecolor{darkpurple}{rgb}{0.47,0.09,0.29}

\patchcmd{\section}{\scshape}{\fontfamily{ppl}\bfseries\scshape\selectfont}{}{%
  \PackageWarning{aiproofs}{Could not bold the section font; amsart may have changed}}

\newcommand{\aiauthor}{%
  \author[Suvrit Sra]{Suvrit Sra\textsuperscript{\normalfont*}}%
  \thanks{*\,TU Munich, Dept.~of Mathematics, School of CIT; Garching, Germany.
    Email: \texttt{s.sra@tum.de}.}}

\makeatother

\hypersetup{colorlinks=true,linkcolor=blue!45!black,citecolor=blue!45!black,urlcolor=red!65!black}
\setlist{nosep,leftmargin=*}

\newenvironment{aiusage}{\section*{AI usage}}{}

\newcommand{\aiacknowledgements}[1][]{%
  \section*{Acknowledgements}
  \if\relax\detokenize{#1}\relax\else #1\ \fi
  The author (SS) thanks the Alexander von Humboldt (AvH) Foundation for
  their generous support via a Humboldt Professorship in AI.%
}

\newcommand{\reals}{\mathbb{R}}

\newcommand{\E}{\mathbb{E}}

\newcommand{\nlsum}{\sum\nolimits}

\newtheorem{theorem}{Theorem}[section]

\newtheorem{prop}[theorem]{Proposition}
\newtheorem{lemma}[theorem]{Lemma}
\newtheorem{corr}[theorem]{Corollary}

\theoremstyle{definition}

\definecolor{cdarkred}{RGB}{180,70,30} %
\definecolor{cdarkgreen}{RGB}{0,130,0} %
\definecolor{azure}{RGB}{240,255,255}
\definecolor{cmidgreen}{RGB}{201,219,216}
\definecolor{clightgreen}{RGB}{200,255,200}
\definecolor{clightgray}{RGB}{240,240,240}
\definecolor{cmygreen}{RGB}{10,120,0}
\definecolor{cmyblue}{RGB}{50,50,255}
\definecolor{sblue}{RGB}{100,100,255}
\definecolor{cmypurple}{RGB}{150,50,180}
\definecolor{cmyred}{RGB}{238, 64,  0}
\definecolor{cdarkblue}{RGB}{30,75,170}
\definecolor{csharpred}{RGB}{255,0,0}

\newcommand{\darkblue}[1]{{\color{cdarkblue} #1}}

\newtcolorbox{dashedbox}[1][]{
  nobeforeafter,
  colback=yellow!10!white,
  boxrule=0pt,
  enhanced jigsaw,
  top=2pt,
  left=2pt,
  right=2pt,
  bottom=0pt,
  borderline horizontal={0.5pt}{0pt}{dashed},
  borderline vertical={0.5pt}{0pt}{dashed},
  #1
}

\newtcbox{\tcbmathbox}[1][]{
  nobeforeafter,
  colback=yellow!30!white,
  boxrule=0pt,
  enhanced jigsaw,
  top=0pt,
  left=1pt,
  right=1pt,
  bottom=0pt,
  borderline horizontal={0.5pt}{0pt}{dashed},
  borderline vertical={0.5pt}{0pt}{dashed},
  #1
}

\newcommand{\Z}{\mathbb{Z}}
\newcommand{\Sc}{\mathcal{S}}
\newcommand{\dd}{\,\mathrm d}
\newcommand{\one}{\mathbf 1}
\allowdisplaybreaks
\hypersetup{pdftitle={A proof of McLeod's 1959 conjecture on complete homogeneous symmetric ratios},
pdfauthor={Suvrit Sra}}

\begin{document}
\title[A proof of McLeod's 1959 conjecture]{A Proof of McLeod's 1959 Conjecture\\
on Complete Homogeneous Symmetric Ratios}
\aiauthor
\begin{abstract}
  We prove McLeod's (1959) conjecture that $(h_{\ell+r}/h_\ell)^{1/r}$ is
  subadditive for all $\ell\ge0$ and $r\ge1$, where $h_k$ is the complete
  homogeneous symmetric polynomial of degree~$k$
  (J.~B.~McLeod. \darkblue{\emph{``On four inequalities in symmetric
      functions.''}}  Proc.~Edinburgh Math.~Soc., 1959).  We first prove that
  the ratio $h_k/h_{k-1}$ is convex, and our proof takes an unexpected
  detour through queuing theory. The geometric mean of $r$ consecutive such
  ratios equals McLeod's conjectured function, but geometric means of convex
  functions need not be convex. Our second and decisive step is a tensorization
  trick that couples neighboring ratios tightly enough to make their geometric
  mean convex.
\end{abstract}
\maketitle

\section{Introduction}
For $x=(x_1,\dots,x_n)\in\reals_+^n$, the complete homogeneous symmetric
polynomial of degree $k$ is
\begin{equation}\label{eq:def-h}
 h_k(x) :=\sum_{1\le i_1\le i_2\le\cdots\le i_k\le n}x_{i_1}x_{i_2}\cdots x_{i_k},
 \qquad k\ge1,\qquad h_0=1.
\end{equation}
\citet[\S1, p.~211]{mcleod1959} proved that $h_r^{1/r}$ is subadditive on $\reals_+^n$ and asked whether $(h_{\ell+r}/h_\ell)^{1/r}$ is also subadditive. Our main result proves McLeod's conjecture in the equivalent form of convexity.

\begin{theorem}[McLeod's inequality]\label{thm:main}
For every $n\ge1$ and all integers $\ell\ge0$ and $r\ge1$, the function
\begin{equation}\label{eq:G-def}
 G_{\ell,r}(x)=\left(\frac{h_{\ell+r}(x)}{h_\ell(x)}\right)^{1/r},
 \qquad G_{\ell,r}(0)=0,
\end{equation}
is convex on $\reals_+^n$. Equivalently (since $G_{\ell,r}$ is positively homogeneous of degree one),
\begin{equation}\label{eq:subadd-main}
 G_{\ell,r}(x+y)\le G_{\ell,r}(x)+G_{\ell,r}(y),
 \qquad x,y\in\reals_+^n.
\end{equation}
\end{theorem}
Beyond McLeod's $\ell=0$ result, the case $\ell=1$ was proved
by~\citet{baston1978two}.  \citet[Theorems~14--15]{sra2020} extended both cases
to coordinatewise powers (see also
Section~\ref{sec:discussion}). Theorem~\ref{thm:main} finally resolves the
general case.  Our analysis also yields a quantitative lower bound for the
directional curvature of $G_{\ell,r}$---see Appendix~\ref{app:remainder}.

\subsection{High-level proof strategy}\label{sec:idea}
Taking cue from the proof of concavity of an analogous ratio of elementary symmetric
polynomials~\citep{marcus1957inequalities,sra2020}, we could consider
\begin{equation}
  \label{eq:1}
  \frac{h_{\ell+r}}{h_\ell} = \frac{h_{\ell+r}}{h_{\ell+r-1}} \times \cdots \times \frac{h_{\ell+1}}{h_\ell}.
\end{equation}
However, we are pursuing convexity \emph{not} concavity---could we still hope that convexity of the individual ratios $\rho_k=h_k/h_{k-1}$ helps us toward proving convexity of the product~\eqref{eq:1}? Curiously, even proving the convexity of the ratio $\rho_k$ turns out to be far from trivial; it is the content of Theorem~\ref{thm:adjacent}, a result whose proof relies on an unexpected connection to queuing.

Convexity of $\rho_k$ by itself is insufficient: $G_{\ell,r}$ is the geometric mean of these ratios, and geometric means of convex functions need not be
convex. This step is precisely where the elementary symmetric argument succeeds and ours stalls. This obstacle forces us to revisit geometric means more closely. While geometric means do not preserve convexity in general, perhaps for the ratios
$\rho_k$ they still do, because neighboring ratios are tightly linked. We use two ingredients to turn this hope into a proof:
\vskip4pt
\begin{enumerate}[label=(\roman*)]
\item work with the \emph{reciprocal ratios} $h_k/h_{k+1}$, and exploit their concavity in $k$;
\item \emph{tensorize an auxiliary variable}, splitting it equally among many
  auxiliary variables; importantly, it yields a bound coupling the curvatures and
  logarithmic slopes of neighboring ratios.
\end{enumerate}

\vskip5pt Concavity in $k$ is already known in queuing theory; our
tensorization trick also recovers an independent proof. Tensorization
rests on the relation $\partial_z h_k(x,z)\big|_{z=0}=h_{k-1}(x)$ for an auxiliary
variable $z$; importantly, such auxiliary variables let neighboring degrees interact. We may
use as many auxiliaries as needed, and in the limit only cross terms survive. A discrete Poincar\'e inequality then combines concavity in $k$
with a bound on curvature to control the geometric mean
(Section~\ref{sec:general}). %

\section{Convexity of consecutive ratios}\label{sec:adjacent-main}
The polynomial $h_k(x)$ normalizes the stationary law of a closed network of
$n$ exponential single-server stations with service means $x_i$ and $k$
circulating customers~\citep{gordon1967closed}. Its reciprocal ratio
$h_{k-1}/h_k$ is the \emph{throughput} at each station (the long-run rate of
service completions).  For a fixed customer count,
$\rho_k=h_k/h_{k-1}$ is therefore the \emph{time per completion}: the
average time between successive service completions at a station. We
use this interpretation to prove convexity of consecutive ratios in the service
means. %

\begin{theorem}[Consecutive-ratio convexity]\label{thm:adjacent}
For every $n\ge1$ and $k\ge1$, the function
\begin{equation}\label{eq:rho-def}
 \rho_k(x)=\frac{h_k(x)}{h_{k-1}(x)},\qquad \rho_k(0)=0,
\end{equation}
is convex and coordinatewise nondecreasing on $\reals_+^n$.
\end{theorem}
The monotonicity is known: since $h_k=s_{(k)}$, it is a special case of
the monotonicity of Schur ratios $s_\lambda/s_\mu$ for $\mu\subseteq\lambda$
\citep[Proposition~8.1]{khare2021}. Our proof also recovers it directly.

\vskip4pt
\noindent\textbf{Note}. We first tried a direct Hessian computation to establish convexity. For  $k\le4$ with some algebra and a Cauchy--Schwarz argument this route succeeds---see Appendix~\ref{sec:direct-hessian}. But for $k\ge5$, the
resulting identity does not seem to admit any direct proof of Hessian positivity. %

\subsection{The queuing theory proof}\label{sec:queue-main}
The reason queuing theory unlocks the proof is almost coincidental: there exists
a queuing system in which $\rho_k$ is the average time between successive
service completions, and in this system the completion times are convex almost
by construction. More specifically, we must find a way to compare completion
times as the service means vary, and to quantify the throughput (i.e., the
long-run rate of completions) to identify their long-run behavior. The max-plus
recursions of~\citet{baccelli1992synchronization} supply the comparison, while
the product-form stationary law of~\citet{gordon1967closed} supplies the
throughput. %

\begin{proof}[\textbf{Proof of Theorem~\ref{thm:adjacent}}]
For $n=1$ the ratio is $x_1$, so assume $n\ge2$ and first take
$x\in\reals_{++}^n$. Put $k$ customers at station~$1$ of a cyclic network $1\to2\to\cdots\to n\to1$. Each station serves one customer at a time, in arrival order, with independent exponential service times of mean $x_i$.

\vskip4pt
\noindent\textbf{\emph{Step 1}}. We begin by proving convexity of completions in the means $x$. To compare networks with different service means, fix independent standard exponentials $E_{i,j}$ and use $x_iE_{i,j}$ as the $j$th service time at station~$i$ for every $x$. Only the deterministic scales $x_i$ vary between networks; each service time is therefore linear in $x$ on this common probability space.

Service begins when both the server and the next customer are available,
so its starting time is the maximum of these two availability times.
Because all customers follow the same route and each station serves them
in arrival order, customer order is preserved. The customer completing the $j$th service at
station $i-1$ is the $j$th arrival to station $i$ for $2\le i\le n$.
At station~$1$, the first $k$ required arrivals are present at time zero;
thereafter the $j$th required arrival is the $(j-k)$th completion at station~$n$.
Writing $D_i(j;x)$ for the time of the $j$th completion at station~$i$,
we therefore obtain
\begin{align}
 D_1(j;x)&=x_1E_{1,j}+\max\{D_1(j-1;x),D_n(j-k;x)\},\notag\\
 D_i(j;x)&=x_iE_{i,j}+\max\{D_i(j-1;x),D_{i-1}(j;x)\},\quad 2\le i\le n,
 \label{eq:recursion-main}
\end{align}
with $D_i(j;x)=0$ for $j\le0$. For each realization of the exponentials, these
recursions use only addition of nonnegative linear functions of $x$ and taking a
maximum.  Both operations preserve convexity and coordinatewise
monotonicity. The return to station~$1$ uses the earlier index $j-k<j$, so the
cyclic route permits induction in the order $(j,1),\dots,(j,n)$.  Starting from
the zero initial times, this induction shows that every $D_i(j;\cdot)$ is convex
and coordinatewise nondecreasing.

\vskip4pt
\noindent\textbf{\emph{Step 2}}. Next, we show that the normalized completion times $D_i(j;x)/j$ converge to $\rho_k(x)$ as $j\to\infty$. To identify the long-run behavior of the completion times, consider the
\emph{queue-length process}. Let $e_i\in\Z_+^n$ denote the $i$th standard basis vector, with indices read cyclically ($e_{n+1}=e_1$). The process has states
$q\in\Z_+^n$ with $\sum_iq_i=k$ and transitions $q\to q-e_i+e_{i+1}$ at rate
$1/x_i$ when $q_i>0$.  Global balance gives the stationary law
$\pi_x(q)=x^q/h_k(x)$, where $x^q:=x_1^{q_1}x_2^{q_2}\cdots x_n^{q_n}$. Incoming
and outgoing fluxes both equal $\pi_x(q)\sum_{i:q_i>0}x_i^{-1}$ (see Appendix~\ref{sec:network}). Consequently, the throughput at any
station is
\begin{equation}\label{eq:rate-main}
 \sum_{q:q_i>0}\frac{\pi_x(q)}{x_i}
 =\frac{1}{h_k(x)}\sum_{q:q_i>0}x^{q-e_i}
 =\frac{h_{k-1}(x)}{h_k(x)} = \frac{1}{\rho_k(x)}.
\end{equation}
By ergodicity, this throughput~\eqref{eq:rate-main} is the long-run number of completions per unit time at each station; inverting the completion count therefore yields
\begin{equation}\label{eq:limit-main}
 \frac{D_i(j;x)}j\longrightarrow\rho_k(x)\qquad\text{almost surely as } j\to\infty.
\end{equation}
Appendix~\ref{sec:limit} proves this limit from the bounded-rate
counting-process martingale.

Finally, since convexity and coordinatewise monotonicity each compare $\rho_k$
at finitely many points, both pass from the completion times to the almost-sure
limit~\eqref{eq:limit-main} on $\reals_{++}^n$. Since $\rho_k$ is
continuous on $\reals_+^n$, both properties extend to $\reals_+^n$.
\end{proof}

\section{The proof of Theorem~\ref{thm:main}: reciprocal ratios and tensorization}\label{sec:general}
The queuing argument proves convexity of every $\rho_k$. To prove convexity of the geometric mean $G_{\ell,r}=\bigl(\prod_{k=\ell+1}^{\ell+r}\rho_k\bigr)^{1/r}$ we develop new ideas below. It suffices to establish convexity of $t \mapsto G_{\ell,r}(x+tv)$ for $x\in \reals_+^n$ and an arbitrary direction $v$ (ensuring that $x+tv \in \reals_+^n$ too). As before we may work with $x>0$. We write $D_v$ for the directional derivative along $v$, and seek to prove that $D_v^2G_{\ell,r}(x)\ge0$. Because $\log G_{\ell,r}$ is the average of the $\log\rho_k$, we describe each factor by its logarithmic slope and
relative curvature,
\begin{equation}\label{eq:slope-curv-def}
 d_k=D_v\log\rho_k,\qquad \kappa_k=\frac{D_v^2\rho_k}{\rho_k},
 \qquad \bar d=\frac1r\sum_{k=\ell+1}^{\ell+r}d_k.
\end{equation}
Twice differentiating the logarithm of $G_{\ell,r}$ (in the geometric mean form above) gives
\begin{equation}\label{eq:telescope-hessian}
 \frac{D_v^2G_{\ell,r}}{G_{\ell,r}}
 =\frac1r\left\{\sum_{k=\ell+1}^{\ell+r}\kappa_k
                 -\sum_{k=\ell+1}^{\ell+r}(d_k-\bar d)^2\right\}.
\end{equation}
Indeed, $D_v^2\log\rho_k=\kappa_k-d_k^2$, and the first derivative of
$\log G_{\ell,r}$ is $\bar d$. Theorem~\ref{thm:adjacent} makes each
$\kappa_k$ nonnegative, but \eqref{eq:telescope-hessian} subtracts the variance of the
logarithmic slopes. Convexity of $G_{\ell,r}$, equivalently nonnegativity of~\eqref{eq:telescope-hessian} therefore requires enough curvature to control this variance.

\subsection{Reciprocal ratios}
To control the variance in~\eqref{eq:telescope-hessian}, we need to know how
the ratios change from one degree to the next. Since~\eqref{eq:telescope-hessian}
is written in logarithms, the first candidate is convexity of
$k\mapsto\log\rho_k$. We work instead with a stronger property: concavity of the
reciprocal ratios
\begin{equation}\label{eq:sigma-def}
 \sigma_k :=\frac{h_{k-1}}{h_k}=\rho_k^{-1}\quad(k\ge1),\qquad \sigma_0=0,
\end{equation}
that is, nonnegativity of their second differences
\begin{equation}\label{eq:alpha-def}
 \alpha_k := 2\sigma_k-\sigma_{k-1}-\sigma_{k+1}.
\end{equation}
Concavity of $\sigma_k$ implies the logarithmic statement, since $\log\sigma_k$ is
then concave and $\log\rho_k=-\log\sigma_k$. Because $\sigma_k$ is the station
throughput (Section~\ref{sec:adjacent-main}), $\alpha_k\ge0$ says that throughput
is concave in the customer count. This concavity is known~\citep{dowdy1984throughput}, and \citet[Theorem~1]{shanthikumar1988second} extended it to networks whose service
rates are nondecreasing and concave in the local queue length, a class that
contains ours (Appendix~\ref{sec:sigma}).

But concavity of $\sigma_k$ alone does not control the logarithmic slopes in
\eqref{eq:telescope-hessian}. We also need to bound their neighboring
differences by curvature in the service means. The tensorization argument below proves this crucial coupled estimate (it also recovers an independent proof of concavity). %

\subsection{Tensorization}\label{sec:amplification}
The coupled estimate comes from augmenting $x$ to the tuple $(x,z)$, where $z$ is an auxiliary, and then computing the Hessian of $\rho_k(x,z)$. Derivatives $\partial_z|_{z=0}$ lower the degree, e.g.,
$\partial_zh_j(x,z)\big|_{z=0}=h_{j-1}(x)$, so the Hessian involves both $\rho_{k-1}$ and $\rho_k$.
Convexity of $\rho_k$ therefore ties the curvature $\kappa_k$~(see~\eqref{eq:slope-curv-def}) to the slope difference $d_k-d_{k-1}$.

A single auxiliary, however, ties them too loosely. With $\tau_k=\rho_k/\rho_{k-1}>0$ ($k\ge2$), its
diagonal entry is $\partial_z^2\rho_k=\tau_k(2\sigma_{k-1}-2\sigma_{k-2})$, a first difference of
$\sigma$, whereas the variance bound of Section~\ref{sec:poincare} needs the second
differences~$\alpha_k$. This limitation is what tensorization addresses. More precisely, we split $z$
into $M$ auxiliaries that share its value equally. We keep track of what this split changes using the
generating function $H_x(t)=\sum_{m\ge0}h_m(x)t^m=\prod_i(1-x_it)^{-1}$; each auxiliary
multiplies it by one factor, and the coefficient of $t^s$ in that factor multiplies $h_{j-s}$. A single
auxiliary with value $u$ contributes $(1-ut)^{-1}=1+ut+u^2t^2+\cdots$, so that
$h_j(x,u)=h_j(x)+u\,h_{j-1}(x)+u^2h_{j-2}(x)+\cdots$. Splitting $u$ equally among $M$ auxiliaries gives instead
\begin{equation}\label{eq:tensor-generating}
 H_{(x,(u/M)\one_M)}(t)
 =H_x(t)\left(1-\frac{ut}{M}\right)^{-M}\quad
 \xrightarrow{M\to\infty}\quad H_x(t)e^{ut}.
\end{equation}
As $M\to\infty$, the split keeps the first-order term $ut$ but halves the second-order term, from $u^2t^2$ to $u^2t^2/2$;
this halving turns $\tau_k(2\sigma_{k-1}-2\sigma_{k-2})$ into
$\tau_k(2\sigma_{k-1}-\sigma_{k-2}-\sigma_k)=\tau_k\alpha_{k-1}$.

\begin{prop}[Coupled curvature]\label{prop:sigma-concave}\label{prop:local-curvature}
For every $k\ge1$, $\alpha_k\ge0$. For $k\ge2$,
\begin{equation}\label{eq:local-curvature}
 \kappa_k\ge\frac{\sigma_{k-1}}{\alpha_{k-1}}(d_k-d_{k-1})^2
 \quad\text{if }\alpha_{k-1}>0;\quad\text{and if}\ \alpha_{k-1}=0, \text{then}\ d_k=d_{k-1}.
\end{equation}
\end{prop}

\begin{proof}
  Let $z=(z_1,\dots,z_M)$ and write $\widetilde\rho_k(x,z)=\rho_k(x_1,\dots,x_n,z_1,\dots,z_M)$.  %
  By Theorem~\ref{thm:adjacent}:
\begin{equation}\label{eq:hess-boundary-psd}
 \nabla^2\widetilde\rho_k(x,0)\succeq0.
\end{equation}
At $(x,0)$, the auxiliary derivatives of $h_j$ are $h_{j-1}$ at first order and $h_{j-2}$ at mixed second order. The quotient rule then gives,
for $a\ne b$,
\begin{align}
 \partial_{z_a}\widetilde\rho_k&=1-\tau_k,\qquad
 \partial_{z_a}\partial_{z_b}\widetilde\rho_k=\tau_k\alpha_{k-1},
 \label{eq:aux-derivs}\\
 D_v\partial_{z_a}\widetilde\rho_k&=-\tau_k(d_k-d_{k-1}).
 \label{eq:mixed-derivative}
\end{align}
Here $h_{-1}=0$; Appendix~\ref{app:auxiliary} gives the full calculation. In particular, the second
formula is
\[
 \partial_{z_a}\partial_{z_b}\widetilde\rho_k
 =\tau_k\left(2\frac{h_{k-2}}{h_{k-1}}
                 -\frac{h_{k-1}}{h_k}-\frac{h_{k-3}}{h_{k-2}}\right),
\]
which explains the choice of reciprocal ratios.

Concavity in $k$ comes from the auxiliary direction $\one_M$ alone. By symmetry, the auxiliary Hessian
block has one common diagonal entry $a_k$ and one common off-diagonal entry $b_k=\tau_k\alpha_{k-1}$.
Neither depends on $M$. Positivity in the direction $\one_M$ gives $a_k+(M-1)b_k\ge0$ for every $M$, and
hence $b_k\ge0$.  Applying this bound at degree $k+1$, where $b_{k+1}=\tau_{k+1}\alpha_k$, proves
$\alpha_k\ge0$.

To bring in the service means, test the Hessian~\eqref{eq:hess-boundary-psd} against
$v$ together with the auxiliary direction, i.e., against $(v,s/M,\dots,s/M)$. Writing
$u_k=-\tau_k(d_k-d_{k-1})$ for the mixed entry~\eqref{eq:mixed-derivative},
the result is
\begin{equation}\label{eq:tensor-quadratic}
 D_v^2\rho_k+2su_k+s^2\left(b_k+\frac{a_k-b_k}{M}\right)\ge0.
\end{equation}
The mixed term $2su_k$ does not depend on $M$, while the diagonal correction tends to zero. Sending
$M\to\infty$ leaves
\begin{equation}\label{eq:quadratic-limit}
 D_v^2\rho_k+2su_k+b_ks^2\ge0\qquad(s\in\reals).
\end{equation}
If $b_k>0$, minimize over $s$ and divide by $\rho_k$ to obtain
\begin{equation}
  \label{eq:2}
  \kappa_k\ge\frac{u_k^2}{\rho_kb_k}
  =\frac{\sigma_{k-1}}{\alpha_{k-1}}(d_k-d_{k-1})^2.
\end{equation}
If $b_k=0$, nonnegativity of the affine expression for all $s$ forces
$u_k=0$, so $d_k=d_{k-1}$. The same tensorization therefore proves both
reciprocal concavity and the curvature estimate.
\end{proof}

\subsection{A finite variance bound}\label{sec:poincare}
We must now combine the lower-bound~\eqref{eq:2} with another inequality that helps us prove non-negativity of~\eqref{eq:telescope-hessian}. Lemma~\ref{prop:poincare} performs this task.

\begin{lemma}[Poincar\'e inequality]\label{prop:poincare}
Fix $\ell\ge0$ and $r\ge2$, and let $J=\{1\le p<r:\alpha_{\ell+p}>0\}$. Let
$w_0,\dots,w_{r-1}\in\reals$ have mean $\bar w$ and increments
$\eta_p=w_p-w_{p-1}$, with $\eta_p=0$ for $p\notin J$. Then
\begin{equation}\label{eq:poincare}
 \nlsum_{j=0}^{r-1}(w_j-\bar w)^2
 \le\nlsum_{p\in J}\frac{R_p}{\alpha_{\ell+p}}\eta_p^2
 \le\nlsum_{p\in J}\frac{\sigma_{\ell+p}}{\alpha_{\ell+p}}\eta_p^2,
\end{equation}
where $R_p\ge0$ is the gap between $\sigma$ and its chord over $[\ell,\ell+r]$, equalling
\begin{equation}\label{eq:finite-row}
 R_p=\sigma_{\ell+p}-\tfrac{r-p}{r}\sigma_\ell-\tfrac{p}{r}\sigma_{\ell+r}.
\end{equation}
\end{lemma}
\begin{proof}
A quick calculation (see Lemma~\ref{lem:green-kernel}) allows us to write the left hand side of~\eqref{eq:poincare} as
\begin{equation}\label{eq:green-variance}
 \nlsum_{j=0}^{r-1}(w_j-\bar w)^2=\nlsum_{p,q\in J}\Gamma_{pq}\eta_p\eta_q,
 \quad\text{where}\quad\Gamma_{pq}=\min(p,q)-pq/r\ge0,
\end{equation}
where the sum runs over $J$ because $\eta_p=0$ for $p\notin J$. Since $\Gamma\ge0$, we can bound~\eqref{eq:green-variance} as follows: for any $\beta_p>0$, multiply $\eta_p\eta_q \le \tfrac12\bigl(\eta_p^2\beta_q/\beta_p+\eta_q^2\beta_p/\beta_q\bigr)$ by
$\Gamma_{pq}$ and sum over $p,q\in J$ to obtain
\[
 \nlsum_{p,q\in J}\Gamma_{pq}\eta_p\eta_q\le\nlsum_{p\in J}c_p\eta_p^2,
 \quad\text{where}\quad c_p=\frac1{\beta_p}\nlsum_{q\in J}\Gamma_{pq}\beta_q.
\]
Set $\beta_p=\alpha_{\ell+p}$ (positive on $J$ and by
Proposition~\ref{prop:local-curvature}, zero outside). Then
$\alpha_{\ell+p}c_p=\sum_{q=1}^{r-1}\Gamma_{pq}\alpha_{\ell+q}$. The matrix
$\Gamma$ inverts the second difference with zero boundary values: every $f$ with
$f_0=f_r=0$ satisfies $f_p=\sum_q\Gamma_{pq}(2f_q-f_{q-1}-f_{q+1})$. Extended
by~\eqref{eq:finite-row} to $p=0,r$, the sequence $R$ has $R_0=R_r=0$ and, since
the chord is linear in $p$, $2R_p-R_{p-1}-R_{p+1}=\alpha_{\ell+p}$. Hence
$\alpha_{\ell+p}c_p=R_p$, which gives the first inequality in~\eqref{eq:poincare};
it also shows $R_p\ge0$, since $\Gamma$ and $\alpha$ are nonnegative. The second
inequality follows from $R_p\le\sigma_{\ell+p}$, which holds because
$\sigma_\ell,\sigma_{\ell+r}\ge0$.
\end{proof}
\noindent\textbf{\emph{Proof of Theorem~\ref{thm:main}}}. For $r=1$ it is
Theorem~\ref{thm:adjacent}. For $r\ge2$, take $w_j=d_{\ell+1+j}$, so that
$\eta_p=d_{\ell+p+1}-d_{\ell+p}$. Proposition~\ref{prop:local-curvature} gives
$\eta_p=0$ for $p\notin J$, which is the hypothesis of Lemma~\ref{prop:poincare},
and it bounds $\kappa_{\ell+p+1}$ for $p\in J$. Since every $\kappa_k\ge0$, dropping $\kappa_{\ell+1}$ and the terms with
$p\notin J$ gives
\begin{equation}
  \label{eq:4}
  \nlsum_{k=\ell+1}^{\ell+r}\kappa_k
  \stackrel{\eqref{eq:2}}{\ge}\nlsum_{p\in J}\frac{\sigma_{\ell+p}}{\alpha_{\ell+p}}\eta_p^2
  \stackrel{\eqref{eq:poincare}}{\ge}\nlsum_{k=\ell+1}^{\ell+r}(d_k-\bar d)^2.
\end{equation}
Using~\eqref{eq:4}, non-negativity of the right side
of~\eqref{eq:telescope-hessian} is immediate, which proves Theorem~\ref{thm:main}.\qed

\section{Minor Extension to Powers}\label{sec:discussion}
\citet[Theorems~14--15]{sra2020} extended the cases $\ell=0$ and $\ell=1$
of Theorem~\ref{thm:main} to coordinatewise powers $x^{\circ p}=(x_1^p,\dots,x_n^p)$. Theorem~\ref{thm:main} already implies the same extension for all $\ell$ and $r$.

\begin{corr}[Power extension]\label{cor:power}
For every $n\ge1$, all integers $\ell\ge0$, $r\ge1$, and every real $p\ge1$,
\begin{equation}\label{eq:Gp-def}
 G^{[p]}_{\ell,r}(x)=
 \left(\frac{h_{\ell+r}(x^{\circ p})}{h_\ell(x^{\circ p})}\right)^{1/(pr)},
 \qquad G^{[p]}_{\ell,r}(0)=0,
\end{equation}
is convex on $\reals_+^n$, or equivalently subadditive. No $0<p<1$ has
this property for all $n,\ell,r$.
\end{corr}

The same product formula that led to the Hessian argument also preserves
coordinatewise monotonicity: each factor $\rho_k$ is nondecreasing by
Theorem~\ref{thm:adjacent}, so $G_{\ell,r}$ is nondecreasing as well.
Together with convexity, this monotonicity allows coordinatewise powers to be inserted
into the triangle inequality. The following lemma isolates that last step.

\begin{lemma}[Monotone power convexification]\label{lem:power-convexify}
Let $F:\reals_+^n\to[0,\infty)$ be convex, positively homogeneous of
degree one, coordinatewise nondecreasing, and positive off the origin.
Then $x\mapsto F(x^{\circ p})^{1/p}$ is convex and one-homogeneous for
every $p\ge1$.
\end{lemma}

\begin{proof}
The transformed function is one-homogeneous, so it suffices to prove
subadditivity. Normalize each vector to lie on its unit level set: for
nonzero $x,y$, put
$a=F(x^{\circ p})^{1/p}$, $b=F(y^{\circ p})^{1/p}$,
$u=x/a$, $v=y/b$, and $\lambda=a/(a+b)$. Thus
$F(u^{\circ p})=F(v^{\circ p})=1$. Coordinatewise Jensen, followed by
monotonicity and convexity of $F$, gives
\[
 F\bigl((\lambda u+(1-\lambda)v)^{\circ p}\bigr)
 \le F\bigl(\lambda u^{\circ p}+(1-\lambda)v^{\circ p}\bigr)\le1.
\]
Since $\lambda u+(1-\lambda)v=(x+y)/(a+b)$, homogeneity gives
$F((x+y)^{\circ p})^{1/p}\le a+b$. If either vector is zero the inequality
is immediate. Subadditivity and homogeneity imply convexity.
\end{proof}

\begin{proof}[Proof of Corollary~\ref{cor:power}]
Apply Lemma~\ref{lem:power-convexify} to $F=G_{\ell,r}$.
For sharpness, take $\ell=0$, $r=1$, $n=2$ and $x=e_1$, $y=e_2$.
For every $0<p<1$,
\[
 G^{[p]}_{0,1}(e_1+e_2)=2^{1/p}>2
   =G^{[p]}_{0,1}(e_1)+G^{[p]}_{0,1}(e_2).
\]
Thus no $p<1$ works uniformly over the family.
\end{proof}

\begin{aiusage}
The results in this paper were developed through extensive use of an AI
assistant (OpenAI's GPT~Pro). None of the claims were obtained zero-shot:
they emerged through sustained back-and-forth discussion, including failed
attempts, counterexamples, audits, and repairs. For example, our first
approach to Theorem~\ref{thm:adjacent} used the exponential-moment
representation $\E S_x^k=k!\,h_k(x)$, where $S_x=\sum_ix_i\xi_i$ with
independent standard exponential $\xi_i$; this representation was key to
the $h_k/h_1$ case in the author's earlier work~\citep{sra2020}. Here it ran
into rough weather, so we cast a wider net: we asked the assistant to
survey where else $h_k$ or the ratio $h_k/h_{k-1}$ arise as probabilities or rates in the literature, and among the
candidates it returned, the queuing interpretation proved the most
fruitful. The tensorization idea was also our own contribution. We used GPT and Claude for assistance with drafting and writing, though the main text is almost entirely human written.
\end{aiusage}

\aiacknowledgements

\appendix

\section{Deferred proofs}\label{app:consecutive}
This appendix supplies the proofs that Sections~\ref{sec:adjacent-main}
and~\ref{sec:general} defer, in the order in which they are used, and ends
with the quantitative remainder mentioned in the introduction.

\subsection{The queuing network}\label{sec:network}\label{sec:limit}

For the queuing proof of Theorem~\ref{thm:adjacent}, this subsection derives
the stationary law of the network and its completion-time limit. Fix throughout integers $n\ge2$ and $k\ge1$,
and let $x\in\reals_{++}^n$. The network of Section~\ref{sec:queue-main} has state space
$\Sc_k=\{q\in\Z_+^n:\sum_{i=1}^nq_i=k\}$, with station indices understood
cyclically.

\begin{prop}[Stationary law]\label{prop:stationary}
The queue-length chain of Section~\ref{sec:queue-main} is irreducible and has
stationary distribution $\pi_x(q)=x^q/h_k(x)$ for $q\in\Sc_k$, where
$x^q=\prod_{i=1}^nx_i^{q_i}$.
\end{prop}

\begin{proof}
Irreducibility is immediate: by repeated service completions, every customer can be moved around the cycle, so every composition of $k$ is reachable from every other one.

The normalization $\sum_{q\in\Sc_k}x^q=h_k(x)$ is exactly~\eqref{eq:def-h}. It
remains to check global balance. The total outgoing flux from a state $q$ is
$\pi_x(q)\sum_{i:q_i>0}1/x_i$. To enter $q$ through a completion at station
$i$, the preceding state must be $q+e_i-e_{i+1}$, which is admissible exactly
when $q_{i+1}>0$. Since $x^{q+e_i-e_{i+1}}=x^qx_i/x_{i+1}$, the associated
incoming flux is $\pi_x(q+e_i-e_{i+1})/x_i=\pi_x(q)/x_{i+1}$. Summing over the
admissible incoming transitions gives
$\pi_x(q)\sum_{i:q_{i+1}>0}1/x_{i+1}=\pi_x(q)\sum_{j:q_j>0}1/x_j$, which is the
outgoing flux. Hence $\pi_x$ is stationary. Uniqueness follows from
finite-state irreducibility.
\end{proof}

To justify the completion-time limit in~\eqref{eq:limit-main}, we first
count service completions. For a fixed station $i$, let $N_i(t;x)$ count the
completions by time $t$, and let $\Theta_k(x)$ be the
throughput~\eqref{eq:rate-main}.

\begin{lemma}[Long-run throughput]\label{lem:ergodic-throughput}
Fix $x\in\reals_{++}^n$. From any initial state of the closed network, every
station $i$ satisfies
\begin{equation}\label{eq:N-limit}
 \frac{N_i(t;x)}{t}\longrightarrow \Theta_k(x)
 =\frac{h_{k-1}(x)}{h_k(x)}
 \qquad\text{almost surely.}
\end{equation}
\end{lemma}

\begin{proof}
The queue-length process $Q(t)$ is an irreducible continuous-time Markov chain on the finite state space $\Sc_k$, hence is positive recurrent and satisfies the ergodic theorem from every initial state.

The counting process at station $i$ has the compensator decomposition
\begin{equation}\label{eq:compensator}
 N_i(t;x)=M_i(t)+\int_0^t\frac{\one_{\{Q_i(s)>0\}}}{x_i}\,\dd s,
\end{equation}
where $M_i$ is a square-integrable martingale.  Its predictable quadratic variation satisfies $\langle M_i\rangle_t\le t/x_i$. To see that $M_i(t)/t\to0$ almost surely, Doob's inequality gives
$\Pr(\sup_{0\le s\le2^m}|M_i(s)|>\varepsilon2^m)\le1/(\varepsilon^2x_i2^m)$,
which is summable in $m$. Borel--Cantelli, followed by interpolation between consecutive dyadic times, proves the claim.  The Markov-chain ergodic theorem gives
\[
 \frac1t\int_0^t\frac{\one_{\{Q_i(s)>0\}}}{x_i}\,\dd s
 \longrightarrow
 \sum_{q:q_i>0}\pi_x(q)\frac1{x_i}
 =\Theta_k(x),
\]
where the last identity is~\eqref{eq:rate-main}.  Substitution into \eqref{eq:compensator} proves \eqref{eq:N-limit}.  The Markov-chain facts used here are standard; see, for example,~\citep{norris1997markov}.
\end{proof}

\begin{lemma}[Inversion of counting processes]\label{lem:inverse-count}
Let $N:[0,\infty)\to\Z_+$ be nondecreasing and right-continuous, with
$N(t)/t\to\theta\in(0,\infty)$. Then $D(r)=\inf\{t:N(t)\ge r\}$ satisfies
$D(r)/r\to1/\theta$.
\end{lemma}

\begin{proof}
Fix $0<\varepsilon<\theta$, and choose $T<\infty$ with
$(\theta-\varepsilon)t\le N(t)\le(\theta+\varepsilon)t$ for $t\ge T$. For all
sufficiently large $r$, the times $t_-=(r-1)/(\theta+\varepsilon)$ and
$t_+=r/(\theta-\varepsilon)$ are at least $T$. The upper estimate on $N$ gives
$N(t_-)\le r-1$, so the $r$th completion has not occurred by $t_-$ and
$D(r)>t_-$. The lower estimate gives $N(t_+)\ge r$, so $D(r)\le t_+$.
Therefore $(r-1)/(r(\theta+\varepsilon))<D(r)/r\le1/(\theta-\varepsilon)$;
let $r\to\infty$ and then $\varepsilon\downarrow0$.
\end{proof}

Combining the two lemmas gives $D_i(r;x)/r\to1/\Theta_k(x)=\rho_k(x)$ almost
surely, which is~\eqref{eq:limit-main}.

\subsection{Derivatives at auxiliary zero coordinates}\label{app:auxiliary}

For Proposition~\ref{prop:local-curvature}, the two lemmas below verify the
derivative formulas~\eqref{eq:aux-derivs} and~\eqref{eq:mixed-derivative}.

\begin{lemma}[Auxiliary derivatives]\label{lem:aux-derivatives}
Let $k\ge2$, and let $a\ne b$ be auxiliary indices. Then
\begin{equation*}
 \frac{\partial\widetilde\rho_k}{\partial z_a}(x,0)=1-\tau_k,
 \qquad
 \frac{\partial^2\widetilde\rho_k}{\partial z_a\,\partial z_b}(x,0)=\tau_k\,\alpha_{k-1}.
\end{equation*}
\end{lemma}

\begin{proof}
For a single auxiliary variable $z$ one has $h_j(x,z)=\sum_{s=0}^jz^sh_{j-s}(x)$, so $\partial_zh_j(x,0)=h_{j-1}(x)$; for two auxiliary variables, the coefficient of $z_az_b$ in $h_j(x,z_a,z_b)$ is $h_{j-2}(x)$, so $\partial^2_{z_az_b}h_j(x,0)=h_{j-2}(x)$.  Here $h_{-1}=0$.  The quotient rule for $\widetilde\rho_k=h_k/h_{k-1}$ at $(x,0)$ gives the
first derivative $(h_{k-1}^2-h_kh_{k-2})/h_{k-1}^2=1-\tau_k$ and,
by~\eqref{eq:alpha-def}, the mixed second derivative
\begin{align*}
 \frac{\partial^2\widetilde\rho_k}{\partial z_a\,\partial z_b}(x,0)
 &=-\frac{h_{k-2}}{h_{k-1}}
   -\frac{h_kh_{k-3}}{h_{k-1}^2}
   +2\,\frac{h_kh_{k-2}^2}{h_{k-1}^3}\\
 &=\tau_k\left(2\,\frac{h_{k-2}}{h_{k-1}}-\frac{h_{k-1}}{h_k}-\frac{h_{k-3}}{h_{k-2}}\right)
 =\tau_k\bigl(2\sigma_{k-1}-\sigma_k-\sigma_{k-2}\bigr)=\tau_k\alpha_{k-1}.
 \qedhere
\end{align*}
\end{proof}

\begin{lemma}[Mixed derivative]\label{lem:mixed-derivative}
Let $k\ge2$, let $a$ be an auxiliary index, and let $D_v$ act on the $x$
variables. Then $D_v\partial_{z_a}\widetilde\rho_k(x,0)=-\tau_k(d_k-d_{k-1})$.
\end{lemma}

\begin{proof}
Mixed partial derivatives of the smooth function $\widetilde\rho_k$ commute,
and $\partial_{z_a}\widetilde\rho_k(\cdot,0)=1-\tau_k$ on $\reals_{++}^n$ by
Lemma~\ref{lem:aux-derivatives}. Since $\tau_k=\rho_k/\rho_{k-1}$, we have
$D_v\tau_k=\tau_k(D_v\log\rho_k-D_v\log\rho_{k-1})=\tau_k(d_k-d_{k-1})$.
\end{proof}

\subsection{The finite variance identity}

Lemma~\ref{prop:poincare} uses the following identity, which writes the
variance of a finite sequence as a quadratic form in its increments.

\begin{lemma}[Finite-path Green kernel]\label{lem:green-kernel}
Let $r\ge1$, and let $w_0,\dots,w_{r-1}\in\reals$ have mean $\bar w$ and
increments $\eta_p=w_p-w_{p-1}$ for $1\le p\le r-1$. Then
\begin{equation}
 \sum_{j=0}^{r-1}(w_j-\bar w)^2
 =\sum_{p,q=1}^{r-1}\Gamma_{pq}\,\eta_p\eta_q,
 \qquad
 \Gamma_{pq}
 =\min(p,q)-\frac{pq}{r}
 =\frac{\min(p,q)\bigl(r-\max(p,q)\bigr)}{r},
\end{equation}
and the matrix $\Gamma$ is positive semidefinite with nonnegative entries.
\end{lemma}

\begin{proof}
The left side of \eqref{eq:green-variance} is unchanged when a constant is added to every $w_j$, so assume $w_0=0$; then $w_j=\sum_{p=1}^{r-1}\one_{\{p\le j\}}\eta_p$.  Let $B$ be the $r\times(r-1)$ matrix $B_{jp}=\one_{\{p\le j\}}$, and let $P=I-r^{-1}\one\one^\top$ be the orthogonal projection onto mean-zero vectors, so that the left side equals $\eta^\top B^\top PB\,\eta$.  Direct counting gives $(B^\top B)_{pq}=r-\max(p,q)$ and
$(B^\top\one)_p=r-p$, hence for $p\le q$,
$(B^\top PB)_{pq}=r-q-(r-p)(r-q)/r=p(r-q)/r=\min(p,q)-pq/r$.
The Gram form $\Gamma=B^\top PB$ gives positive semidefiniteness, and the
product formula gives nonnegative entries.
\end{proof}

\subsection{A quantitative remainder}
\label{app:remainder}
The first inequality in~\eqref{eq:poincare} uses the weights
$R_p\le\sigma_{\ell+p}$; keeping the difference gives a lower bound on the
directional curvature of $G_{\ell,r}$. For $r\ge2$, with $\eta_p=d_{\ell+p+1}-d_{\ell+p}$ and
$J$ as in Lemma~\ref{prop:poincare}, equations
\eqref{eq:poincare} and~\eqref{eq:telescope-hessian} give
\begin{equation}\label{eq:curvature-remainder}
 \frac{D_v^2G_{\ell,r}}{G_{\ell,r}}
 \ge\frac{\kappa_{\ell+1}}r+
 \frac1{r^2}\sum_{p\in J}
 \frac{(r-p)\sigma_\ell+p\sigma_{\ell+r}}{\alpha_{\ell+p}}\eta_p^2.
\end{equation}
Indeed, the curvature bound~\eqref{eq:2} contributes
$\sigma_{\ell+p}\eta_p^2/\alpha_{\ell+p}$, whereas the variance bound
subtracts only $R_p\eta_p^2/\alpha_{\ell+p}$, and $\sigma_{\ell+p}-R_p$ is the
chord term in~\eqref{eq:finite-row}. We do not claim optimality.

\section{Direct Hessian proofs (up to $k=4$) and other minor results}\label{sec:consequences}

The results of this appendix are not needed for Theorem~\ref{thm:main}
and Corollary~\ref{cor:power}. They give a direct Hessian proof of convexity
of $\rho_k$ in low degrees and record how $\rho_k$ and $\sigma_k$ depend on
the degree~$k$. Both parts use the elementary bounds
\begin{equation}\label{eq:basic-main}
 \max_i x_i\le\rho_k(x)\le\sum_i x_i
 \qquad(x\in\reals_+^n\setminus\{0\}).
\end{equation}
Coefficient comparison proves both. For each $j$ the polynomial
$x_jh_{k-1}(x)$ consists precisely of the
monomials of $h_k(x)$ that contain $x_j$, each with coefficient one, so
$h_k(x)-x_jh_{k-1}(x)=h_k(x_1,\dots,x_{j-1},0,x_{j+1},\dots,x_n)\ge0$, which
is the lower bound. For the upper bound, the coefficient of a degree-$k$
monomial $x_{i_1}\cdots x_{i_k}$ in $h_1h_{k-1}$ is the number of distinct
indices among $i_1,\dots,i_k$, hence at least one, so $h_k\le h_1h_{k-1}$
coefficientwise.

\subsection{Direct Hessian proofs in low degrees}\label{sec:direct-hessian}
We compute the Hessian of a consecutive ratio directly. The resulting
identity holds in every degree, and in low degrees it proves convexity.

For this calculation, it is convenient to write
$F_k=h_{k+1}/h_k=\rho_{k+1}$, with $k\ge1$. Fix $x\in\reals_{++}^n$.
Since $F_k$ is homogeneous of degree one, Euler's identity gives
$\nabla^2F_k(x)x=0$. Thus subtracting a multiple of $x$ from a direction $v$ does not change
$D_v^2F_k$. We choose that multiple to make $D_vh_k=0$: explicitly,
replace $v$ by $v-(D_vh_k)/(kh_k)\,x$. This centering removes the
first-derivative terms from the quotient rule, leaving
\begin{equation}\label{eq:centered-quotient}
 h_k^2D_v^2F_k=h_kD_v^2h_{k+1}-h_{k+1}D_v^2h_k.
\end{equation}
To expand the difference on the right, define
\begin{equation}\label{eq:a-c-def}
 a_j=\sum_i v_ix_i^j,\qquad c_j=\sum_i v_i^2x_i^j,
 \qquad H_x(t)=\sum_{m\ge0}h_m(x)t^m=\prod_i(1-x_it)^{-1}.
\end{equation}
Differentiating the generating function gives
$D_vH_x=H_x\sum_{j\ge0}a_jt^{j+1}$. Coefficient comparison therefore yields
$D_vh_m=\sum_{j=0}^{m-1}a_jh_{m-j-1}$, so the centering condition is
\begin{equation}\label{eq:algebraic-centering}
 \sum_{j=0}^{k-1}a_jh_{k-j-1}=0.
\end{equation}
A second differentiation, followed by coefficient comparison, gives
\begin{align}
 D_v^2H_x&=H_x\left[
 \left(\sum_{j\ge0}a_jt^{j+1}\right)^2
 +\sum_{j\ge0}(j+1)c_jt^{j+2}\right],\label{eq:second-derivative-H}\\
 D_v^2h_m&=
 \sum_{i+j\le m-2}a_i a_jh_{m-i-j-2}
 +\sum_{j=0}^{m-2}(j+1)c_jh_{m-j-2},\label{eq:second-derivative-h}
\end{align}
where $i,j\ge0$ in the first sum.
The same quadratic expressions occur at degrees $k$ and $k+1$.
Their coefficients in~\eqref{eq:centered-quotient} are
$\Delta_q=h_kh_{k-q-1}-h_{k+1}h_{k-q-2}$ for $0\le q<k$, with $h_{-1}=0$, and
Jacobi--Trudi identifies each coefficient as a Schur
polynomial~\citep{macdonald1995}:
\begin{equation}\label{eq:Delta-Schur}
 \Delta_q=s_{(k,k-q-1)}(x)\ge0.
\end{equation}
Here $s_\lambda$ denotes the Schur polynomial indexed by $\lambda$;
its nonnegative coefficients give the inequality. Grouping terms by
$q=i+j$ now gives the Hessian formula.

\begin{prop}[Centered Hessian identity]\label{prop:hessian-identity}
Under~\eqref{eq:algebraic-centering},
\begin{equation}\label{eq:all-k-hessian}
 h_k^2D_v^2F_k=
 \sum_{q=0}^{k-1}\Delta_q
 \left(\sum_{j=0}^q a_j a_{q-j}+(q+1)c_q\right).
\end{equation}
\end{prop}

The coefficients $\Delta_q$ are nonnegative, but the quadratic expressions
they multiply contain signed cross terms. The centering relation can
control these cross terms in low degrees, giving the following direct
proofs without using the queuing representation.

\begin{prop}[Direct algebraic convexity]\label{prop:k2-warmup}\label{prop:k3-warmup}
The functions $F_1=h_2/h_1$, $F_2=h_3/h_2$, and $F_3=h_4/h_3$ are convex
on $\reals_+^n$, with value $0$ at the origin.
\end{prop}

\begin{proof}
For $k=1$, centering gives $a_0=0$, and~\eqref{eq:all-k-hessian}
reduces to $h_1^2D_v^2F_1=h_1c_0\ge0$.

For $k=2$, centering eliminates $a_1$: the relation
$h_1a_0+a_1=0$ gives $a_1=-h_1a_0$. With
$\Delta_0=h_1h_2-h_3\ge0$ and $\Delta_1=h_2$,
identity~\eqref{eq:all-k-hessian} becomes
\[
 h_2^2\,D^2F_2[v,v]
 =\Delta_0(a_0^2+c_0)+h_2(2a_0a_1+2c_1)
 =\Delta_0c_0+2h_2c_1-(h_1h_2+h_3)\,a_0^2.
\]
Cauchy--Schwarz gives $a_1^2=\bigl(\sum_ix_iv_i\bigr)^2\le h_1c_1$, hence $c_1\ge h_1a_0^2$ and $h_2^2\,D^2F_2[v,v]\ge\Delta_0(c_0+a_0^2)\ge0$.

At $k=3$, the same strategy leaves a quadratic form in $a_0,a_1$.
Write $s=h_1$ and $p_2=\sum_ix_i^2=2h_2-s^2$; then $\Delta_0=h_3h_2-h_4s$,
$\Delta_1=h_3s-h_4$ and $\Delta_2=h_3$ are nonnegative
by~\eqref{eq:Delta-Schur}. Centering now eliminates
$a_2=-h_2a_0-sa_1$, and~\eqref{eq:all-k-hessian} expands to
\[
 h_3^2\,D^2F_3[v,v]
 =\Delta_0c_0+2\Delta_1c_1+3h_3c_2
 +h_3a_1^2-2h_4a_0a_1-(h_3h_2+h_4s)\,a_0^2.
\]
Three Cauchy--Schwarz inequalities give $c_0\ge a_0^2/n$, $c_1\ge a_1^2/s$ and
$c_2\ge a_2^2/p_2=(h_2a_0+sa_1)^2/p_2$.
Substitution leaves a quadratic form that completes to a square.
Using $p_2+s^2=2h_2$ and the definitions of $\Delta_0,\Delta_1$ gives
\[
 h_3^2\,D^2F_3[v,v]
 \ge\left(h_3+\frac{2\Delta_1}{s}+\frac{3h_3s^2}{p_2}\right)
 \left(a_1+\frac{s}{2}a_0\right)^{\!2}
 +\frac{n+2}{2n}\,\Delta_0\,a_0^2\ \ge\ 0,
\]
which proves the claim for $F_3$.

For each of the three ratios, undoing the radial centering leaves the
Hessian quadratic form unchanged. The bound $0\le F_k\le h_1$
from~\eqref{eq:basic-main} gives continuity at the origin;
elsewhere the denominator is positive. Convexity therefore extends from
the positive orthant to its boundary.
\end{proof}

At higher degrees, the identity still describes the full Hessian, but it
does not by itself establish its sign. The queuing proof of
Section~\ref{sec:queue-main} proves the required positivity for all $k$ simultaneously.

\subsection{Dependence on the degree}\label{sec:sigma}

The queuing interpretation makes the degree index natural: $k$ is the number
of circulating customers. This subsection shows that $\sigma_k$ is
nondecreasing and concave in $k$, and that $\rho_k$ decreases to
$\max_ix_i$.

\begin{lemma}[Degree log-concavity]\label{lem:degree-logconcave}
For every $x\in\reals_+^n\setminus\{0\}$ and every $k\ge1$,
\begin{equation}\label{eq:jacobi-trudi}
 h_k(x)^2-h_{k-1}(x)h_{k+1}(x)=s_{(k,k)}(x)\ge0.
\end{equation}
Consequently $\sigma_k(x)$ is nondecreasing in $k$, and
$0\le\sigma_k(x)\le(\max_ix_i)^{-1}$ by~\eqref{eq:basic-main}.
\end{lemma}

\begin{proof}
The Jacobi--Trudi identity~\citep{macdonald1995} expresses the left side of \eqref{eq:jacobi-trudi} as the Schur polynomial $s_{(k,k)}$, which has nonnegative coefficients.  Dividing by $h_kh_{k+1}>0$ gives $\sigma_{k+1}\ge\sigma_k$ for $k\ge1$, and $\sigma_1\ge0=\sigma_0$ is trivial.  The upper bound is the lower bound in~\eqref{eq:basic-main}.
\end{proof}

\paragraph{Equilibrium-rate representation.}
For $x\in\reals_{++}^n$, choose $0<a<(\max_i x_i)^{-1}$ and let
$Y_1,\dots,Y_n$ be independent geometric random variables with
$\Pr(Y_i=q)=(1-ax_i)(ax_i)^q$ for $q\ge0$. Writing $Y=\sum_iY_i$,
convolution gives $\Pr(Y=k)=a^kh_k(x)/H_x(a)$ for $k\ge0$.
The equilibrium rate of a random variable $Y$ with full support on $\Z_+$
is $r_Y(0)=0$ and $r_Y(k)=\Pr(Y=k-1)/\Pr(Y=k)$ for $k\ge1$
\citep[\S2, equation~(2)]{shanthikumar1988second}. Thus $r_Y(k)=\sigma_k/a$
for $k\ge0$, while $r_{Y_i}(0)=0$ and $r_{Y_i}(k)=1/(ax_i)$ for $k\ge1$, so
each $r_{Y_i}$ is nondecreasing and concave.
\citet[Proposition~1(A)]{shanthikumar1988second} prove that these properties
of equilibrium rates are preserved under convolution. Iteration therefore
gives the nondecreasing concavity of $\sigma_k$ on $\Z_+$. In their cyclic
network representation, choosing the common visit weight $a$ gives exactly
the rate $r_Y=\sigma/a$; this choice accounts for the throughput normalization.

\begin{prop}\label{prop:K-monotonicity}
For fixed $x\in\reals_+^n\setminus\{0\}$, $\rho_k(x)$ is nonincreasing in $k$
and tends to $\max_ix_i$.
\end{prop}

\begin{proof}
Lemma~\ref{lem:degree-logconcave} gives $h_k^2\ge h_{k-1}h_{k+1}$, so
$\rho_{k+1}=h_{k+1}/h_k\le h_k/h_{k-1}=\rho_k$. Hence $\rho_k(x)$ decreases to
some limit $L\ge\max_i x_i$ by~\eqref{eq:basic-main}. Let $M=\max_i x_i$.
Since $h_k$ has $\binom{n+k-1}{k}$ monomials,
$M^k\le h_k(x)\le\binom{n+k-1}{k}M^k$, and therefore $h_k(x)^{1/k}\to M$. But
$h_k(x)^{1/k}=\bigl(\prod_{j=1}^k\rho_j(x)\bigr)^{1/k}$ is a geometric mean of
a positive sequence converging to $L$, so it also converges to $L$. Hence
$L=M$.
\end{proof}

In queuing language, increasing the number of circulating customers decreases the time per completion and drives the system toward the bottleneck service time.


\begin{thebibliography}{11}
\providecommand{\natexlab}[1]{#1}
\providecommand{\url}[1]{\texttt{#1}}
\expandafter\ifx\csname urlstyle\endcsname\relax
  \providecommand{\doi}[1]{doi: #1}\else
  \providecommand{\doi}{doi: \begingroup \urlstyle{rm}\Url}\fi

\bibitem[Baccelli et~al.(1992)Baccelli, Cohen, Olsder, and
  Quadrat]{baccelli1992synchronization}
F.~Baccelli, G.~Cohen, G.~J. Olsder, and J.-P. Quadrat.
\newblock \emph{Synchronization and Linearity: An Algebra for Discrete Event
  Systems}.
\newblock Wiley, Chichester, 1992.

\bibitem[Baston(1978)]{baston1978two}
V.~J. Baston.
\newblock Two inequalities for the complete symmetric functions.
\newblock \emph{Mathematical Proceedings of the Cambridge Philosophical
  Society}, 84\penalty0 (1):\penalty0 1--3, 1978.
\newblock \doi{10.1017/S0305004100054803}.

\bibitem[Dowdy et~al.(1984)Dowdy, Eager, Gordon, and
  Saxton]{dowdy1984throughput}
L.~W. Dowdy, D.~L. Eager, K.~D. Gordon, and L.~V. Saxton.
\newblock Throughput concavity and response time convexity.
\newblock \emph{Information Processing Letters}, 19\penalty0 (4):\penalty0
  209--212, 1984.
\newblock \doi{10.1016/0020-0190(84)90087-5}.

\bibitem[Gordon and Newell(1967)]{gordon1967closed}
W.~J. Gordon and G.~F. Newell.
\newblock Closed queuing systems with exponential servers.
\newblock \emph{Operations Research}, 15\penalty0 (2):\penalty0 254--265, 1967.
\newblock \doi{10.1287/opre.15.2.254}.

\bibitem[Khare and Tao(2021)]{khare2021}
A.~Khare and T.~Tao.
\newblock On the sign patterns of entrywise positivity preservers in fixed
  dimension.
\newblock \emph{American Journal of Mathematics}, 143\penalty0 (6):\penalty0
  1863--1929, 2021.
\newblock \doi{10.1353/ajm.2021.0049}.

\bibitem[Macdonald(1995)]{macdonald1995}
I.~G. Macdonald.
\newblock \emph{Symmetric Functions and Hall Polynomials}.
\newblock Oxford University Press, Oxford, 2nd edition, 1995.

\bibitem[Marcus and Lopes(1957)]{marcus1957inequalities}
M.~Marcus and L.~Lopes.
\newblock Inequalities for symmetric functions and {H}ermitian matrices.
\newblock \emph{Canadian Journal of Mathematics}, 9:\penalty0 305--312, 1957.
\newblock \doi{10.4153/CJM-1957-037-9}.

\bibitem[McLeod(1959)]{mcleod1959}
J.~B. McLeod.
\newblock On four inequalities in symmetric functions.
\newblock \emph{Proceedings of the Edinburgh Mathematical Society}, 11\penalty0
  (4):\penalty0 211--219, 1959.
\newblock \doi{10.1017/S0013091500021921}.

\bibitem[Norris(1997)]{norris1997markov}
J.~R. Norris.
\newblock \emph{{M}arkov Chains}.
\newblock Cambridge Series in Statistical and Probabilistic Mathematics.
  Cambridge University Press, Cambridge, 1997.

\bibitem[Shanthikumar and Yao(1988)]{shanthikumar1988second}
J.~G. Shanthikumar and D.~D. Yao.
\newblock Second-order properties of the throughput of a closed queueing
  network.
\newblock \emph{Mathematics of Operations Research}, 13\penalty0 (3):\penalty0
  524--534, 1988.
\newblock \doi{10.1287/moor.13.3.524}.

\bibitem[Sra(2020)]{sra2020}
S.~Sra.
\newblock New concavity and convexity results for symmetric polynomials and
  their ratios.
\newblock \emph{Linear and Multilinear Algebra}, 68\penalty0 (5):\penalty0
  1031--1038, 2020.
\newblock \doi{10.1080/03081087.2018.1527891}.

\end{thebibliography}
\end{document}